\documentclass[11pt]{amsart}

\usepackage{xypic}
\xyoption{all}
\usepackage{amstext}
\usepackage{amsthm}
\usepackage{amsopn}
\usepackage{amsfonts}
\usepackage{amsmath}
\usepackage{latexsym}
\usepackage{amscd}
\usepackage{amssymb}
\usepackage{amsmath}
\usepackage{verbatim}

\swapnumbers
\newtheorem{theorem}[subsection]{Theorem}

\newtheorem{proposition}[subsection]{Proposition}
\newtheorem{lemma}[subsection]{Lemma}
\newtheorem{corollary}[subsection]{Corollary}

\theoremstyle{definition}

\newtheorem{remark}[subsection]{Remark}
\newtheorem{example}[subsection]{Example}

\numberwithin{equation}{subsection}

\def\gm{\mathfrak m}

\def\mm{\mathfrak m}

\def\dg{\mathbf g}

\def\dh{\mathbf h}
\def\dt{\mathbf t}

\def\dH{\mathbf H}
\def\dG{\mathbf G}

\begin{document}
\title[]{A cohomological proof of Peterson-Kac's theorem on conjugacy
of Cartan subalgebras
for affine Kac--Moody Lie algebras}

\author{V. Chernousov}
\address{Department of Mathematics, University of Alberta,
    Edmonton, Alberta T6G 2G1, Canada}
\thanks{ V. Chernousov was partially supported by the Canada Research
Chairs Program
and an NSERC research grant} \email{chernous@math.ualberta.ca}

\author{V.  Egorov}
\address{Department of Mathematics, University of Alberta,
    Edmonton, Alberta T6G 2G1, Canada}\email{uyahorau@math.ualberta.ca}
\author{P. Gille}
\address{UMR 8553 du CNRS, \'Ecole Normale Sup\'erieure, 45
rue d'Ulm, 75005 Paris, France.} \email{gille@ens.fr}

\author{A. Pianzola}
\address{Department of Mathematics, University of Alberta,
    Edmonton, Alberta T6G 2G1, Canada.
    \newline
 \indent Centro de Altos Estudios en Ciencia Exactas, Avenida de Mayo 866, (1084) Buenos Aires, Argentina.}
\thanks{A. Pianzola wishes to thank NSERC and CONICET for their
continuous support}\email{a.pianzola@math.ualberta.ca}


\begin{abstract}
\noindent This paper
deals with the problem of conjugacy of Cartan subalgebras for
affine Kac-Moody Lie algebras. Unlike the methods used by Peterson
and Kac, our approach
is entirely cohomological and geometric. It is deeply rooted on the theory
of reductive group schemes developed by Demazure and Grothendieck, and on
the work of J. Tits on buildings.
\end{abstract}

\maketitle


\section{Introduction} Chevalley's theorem on the conjugacy of split Cartan
subalgebras is one of the cornerstones of the theory of simple finite
dimensional Lie algebras over a field of characteristic $0.$ Indeed, this
theorem affords the most elegant proof that the root system is an invariant
of the Lie algebra.

The analogous result for symmetrizable Kac-Moody Lie algebras is the
celebrated theorem of Peterson and Kac \cite{PK} (see also \cite{Kmr}
and \cite{MP} for detailed proofs). Beyond the finite dimensional case,
by far the most important Kac-Moody Lie algebras are the affine ones.
These algebras sit at the ``border" of finite dimensional Lie theory and
they can in fact be viewed as ``finite dimensional" (not over the base field
but over a Laurent polynomial ring) in the sense of \cite{SGA3}. This
approach begs the question as to whether an SGA-inspired proof of conjugacy
exists in the affine case. This paper, which builds in \cite{CGP} and
\cite{GP3}, shows that the answer is yes. More precisely, in
\cite{P1} (the untwisted case) and \cite{CGP} (general case) conjugacy is
established for loop algebras by purely Galois cohomological methods. The
step that is missing is extending this result to the ``full" Kac-Moody Lie
algebra. The central extension presents of course no difficulties, but the
introduction of the derivation does. The present paper addresses this issue
thus yielding a new cohomological proof of the conjugacy theorem of Peterson
and Kac in the case of affine Kac-Moody Lie algebras.



\section{Affine Kac-Moody Lie algebras}\label{affinerealization}

\noindent
{\bf Split case.}  Let $\dg$ be a split simple finite dimensional Lie
algebra over an algebraically closed field $k$ of characteristic $0$
and let ${\rm {\bf Aut}}(\dg)$ be its automorphism group.
If $x,y\in \dg$ we denote their product in $\dg$ by $[x,y]$.
We also let $R=k[t^{\pm 1}]$, 
and $L(\dg)=\dg \otimes_k R$. We still denote the Lie product in
$L(\dg)$ by $[x,y]$ where $x,y\in L(\dg)$.

The main object under consideration in this paper is the affine
(split or twisted) Kac-Moody Lie algebra $\widehat{L}$
corresponding to $\dg$.  Any split affine Kac-Moody Lie algebra is of
the form (see \cite{Kac})
$$ \widehat{L} = \dg\otimes_k R \oplus k\,c \oplus k\,d.$$
The element $c$ is central and $d $ is a degree derivation for a natural
grading of $L(\dg)$: if $x\in \dg$ and $p\in\mathbb{Z}$ then
$$
[d,x\otimes t^p]_{\widehat{L}}=p\,x\otimes t^p.
$$

If $l_1=x\otimes t^p,\ l_2=y\otimes t^q\in L(\dg)$ are viewed as elements
in $\widehat{L}$  their Lie product is given by
$$
[x\otimes t^p, y\otimes t^q]_{\widehat{L}}=[x,y]\otimes t^{p+q}+
p\,\langle\,x,y\,\rangle\,  \delta_{0,p+q} \cdot c
$$
where $\langle\,x,y\,\rangle$ is the Killing form on $\dg$ and
$\delta_{0,p+q}$ is Kronecker's delta.


\smallskip

\noindent
{\bf Twisted case.} Let $m$ be a positive integer. Let
$S=k[t^{\pm \frac{1}{m}}]$ be the
ring of Laurent polynomials in the variable $s=t^{\frac{1}{m}}$
with coefficients in $k$. Let
$$L(\dg)_S = L(\dg)\otimes_RS$$ be the
Lie algebra  obtained from the $R$-Lie algebra $L(\dg)$ by the base
change $R\rightarrow S$.
Similarly we define Lie algebras
$$\widetilde{L}(\dg)_S=L(\dg)_S\oplus kc \mbox{\ \ and\ \ }
\widehat{L}(\dg)_S=L(\dg)_S\oplus kc\oplus kd.\footnote{Unlike $L(\dg)_S$,
these object exist over $k$ but not over $S.$}
$$

Fix 
a primitive root of unity $\zeta\in
k$ of degree $m$. The $R$-automorphism
$\zeta^{\times}:S\rightarrow S$ given by $s\mapsto \zeta s$
generates the Galois group $\Gamma={\rm Gal}(S/R)$ which we may identify with
the abstract group $\mathbb{Z}/m\mathbb{Z}$
by means of $\zeta^{\times}$. Note that $\Gamma$ acts naturally on
$
{\rm
{\bf Aut}}(\dg)(S)=
{\rm Aut}_{S-Lie}(L(\dg)_S)$
and on $L(\dg)_S=L(\dg)\otimes_R S$
through the second factor.

Next, let $\sigma$ be an automorphism of $\dg$ of order $m$. This
gives rise to an $S$-automorphism of $L(\dg)_S$ via $x\otimes
s\mapsto \sigma(x)\otimes s$ for $x\in \dg,\, s\in S.$  It then
easily follows that the assignment
$$
    \overline{1}\mapsto z_{\overline{1}}=\sigma^{-1}\in {\rm
    Aut}_{S-Lie}(L(\dg)_S)
$$
gives rise to a cocycle $z=(z_{\overline{i}})\in Z^1(\Gamma,{\rm
Aut}_{S-Lie}(L(\dg)_S))$.
This cocycle, in turn, gives rise to a twisted action of $\Gamma$
on $L(\dg)_S.$
Applying Galois descent
formalism we then obtain the $\Gamma$-invariant subalgebra
$$
    L(\dg,\sigma):=(L(\dg)_S)^{\Gamma}=(L(\dg)\otimes_R
    S)^{\Gamma}.
$$
This is a ``simple Lie algebra over $R$" in the sense of \cite{SGA3}, which
is a twisted form of the
``split simple" $R$-Lie algebra $L(\dg)=\dg\otimes_k R.$ Indeed $S/R$ is an
\'etale extension and from properties of Galois descent we have
$$
    L(\dg,\sigma)\otimes_R S\simeq L(\dg)_S=(\dg\otimes_k
    R)\otimes_R S.
$$
Note that $L(\dg,id)=L(\dg).$

For
$\overline{i}\in \mathbb{Z}/m\mathbb{Z},$ consider the eigenspace
$$\dg_{\overline {i}}=\{x\in \dg: \sigma(x)=\zeta^ix \}.$$
Simple computations show that
$$
L(\dg,\sigma)=\bigoplus_{i\in \mathbb{Z}}
\dg_{\overline{i}}\otimes k[t^{\pm 1}] s^i.
$$

Let
$$\widetilde{L}(\dg,\sigma):=L(\dg,\sigma)\oplus kc \mbox{\ \ and\ \ }
\widehat{L}(\dg,\sigma):=L(\dg,\sigma)\oplus kc\oplus kd.$$
We give $\widehat{L}(\dg,\sigma)$ a Lie algebra structure such
that $c$ is central element, $d$ is the degree derivation, i.e. if
$x\in \dg_{\overline{i}}$ and $p\in \mathbb{Z}$ then
\begin{equation}\label{newbracket}
[d,x\otimes t^{\frac{p}{m}}]:=px\otimes t^\frac{p}{m}
\end{equation}
and if $y\otimes t^{\frac{q}{m}}\in L(\dg,\sigma)$ we get
$$
[x\otimes t^{\frac{p}{m}}, y\otimes
t^{\frac{q}{m}}]_{\widehat{L}(\dg,\sigma)}=[x,y]\otimes
t^{\frac{p+q}{m}}+ p\,\langle\,x,y\,\rangle\,  \delta_{0,p+q}
\cdot c,
$$
where, as before, $\langle\,x,y\,\rangle$ is the Killing form on
$\dg$ and $\delta_{0,p+q}$ is  Kronecker's delta.

\begin{remark} Note that
the Lie algebra
structure on $\widehat{L}(\dg,\sigma)$ is induced by that of on
$\widehat{L}(\dg)_S$
if we view $\widehat{L}(\dg,\sigma)$ as a subset of $\widehat{L}(\dg)_S$.
\end{remark}
\begin{remark} Let $\widehat{\sigma}$ be an automorphism
of $\widehat{L}(\dg)_S$ such that $\widehat{\sigma}|_{L(\dg)_S}=\sigma$,
$\widehat{\sigma}(c)=c, \ \widehat{\sigma}(d)=d$. Then
$\widehat{L}(\dg,\sigma)=
(\widehat{L}(\dg)_S)^{\widehat{\sigma}}$.
\end{remark}

\noindent
{\bf Realization Theorem.} {\it
{\rm (a)} The Lie algebra $\widehat{L}(\dg,\sigma)$ is an affine Kac-Moody
Lie algebra, and every affine Kac-Moody Lie algebra is isomorphic
to some $\widehat{L}(\dg,\sigma)$.

\smallskip

\noindent
{\rm (b)}
$\widehat{L}(\dg,\sigma)\simeq\widehat{L}(\dg,\sigma^{'})$ where
$\sigma^{'}$ is a diagram automorphism with respect to some Cartan
subalgebra of $\dg$.}

\begin{proof}
See \cite[Theorems $7.4,\,8.3$ and $8.5$]{Kac}.
\end{proof}

Let $\phi\in {\rm Aut}_{k-Lie}(\widehat{L}(\dg)_S)$. Since
$\widetilde{L}(\dg)_S$ is the
derived subalgebra of $\widehat{L}(\dg)_S$ the restriction
$\phi|_{\widetilde{L}(\dg)_S}$
induces a $k$-Lie automorphism of $\widetilde{L}(\dg)_S$. Furthermore,
passing to
the quotient $\widetilde{L}(\dg)_S/kc\simeq L(\dg)_S$ the automorphism
$\phi|_{\widetilde{L}(\dg)_S}$ induces an automorphism
of $L(\dg)_S$. This yields a well-defined morphism $${\rm Aut}_{k-Lie}
(\widehat{L}(\dg)_S)\to {\rm Aut}_{k-Lie}(L(\dg)_S).$$ Similar
considerations apply to
${\rm Aut}_{k-Lie}(\widehat{L}(\dg,\sigma))$. The aim of the next few sections
is to show that
these two morphisms are surjective.


\section{$S$-automorphisms of $L(\dg)_S$}

In this section we
construct a ``simple'' system of generators of the automorphism
group $$
{\rm  {\bf Aut}}(\dg)(S)=
{\rm Aut}_{S-Lie}(L(\dg)_S)$$ 
which
can be easily extended to $k$-automorphisms of $\widehat{L}(\dg)_S$. We
produce our list of generators based on a well-known fact that
the group in question is generated by $S$-points of the
corresponding split simple adjoint algebraic group and automorphisms of the
corresponding Dynkin diagram.

More precisely, let $\dG$ be the split simple
simply connected  group over $k$ corresponding to $\dg$ and let
$\overline{\dG}$ be the corresponding adjoint group.
Choose a maximal split $k$-torus
${\mathbf T}\subset \dG$ and denote its image in  $\overline{\dG}$ by
$\overline{\mathbf T}$.
The Lie algebra of ${\mathbf T}$ is a Cartan subalgebra $\dh\subset \dg$.
We fix a Borel subgroup ${\mathbf T}\subset{\mathbf  B}\subset \dG$.

Let $\Sigma = \Sigma(\dG,{\mathbf T})$ be the root system of $\dG$ relative to
${\mathbf T}$.
The Borel subgroup ${\mathbf B}$ determines an ordering of $\Sigma$, hence
the system of simple roots $\Pi = \{\alpha_1,\ldots,\alpha_n \}$.
Fix a Chevalley basis \cite{St67}
$$
\lbrace H_{\alpha_1},\ldots H_{\alpha_n},\ X_{\alpha},\ \alpha\in\Sigma\rbrace
$$
of  $\dg$ corresponding to the pair $({\mathbf T},{\mathbf B})$.
This basis is unique up to signs
and automorphisms of $\dg$ which preserve ${\mathbf B}$ and ${\mathbf T}$ (see
\cite[\S 1, Remark $1$]{St67}).

Since $S$ is a Euclidean ring, by Steinberg \cite{St62} the group $\dG(S)$
is generated by the
so-called root
subgroups $U_{\alpha} = \langle x_{\alpha}(u)\mid u \in S\rangle$,
where $\alpha\in\Sigma$ and
\begin{equation}\label{exponent}
x_{\alpha}(u)=\exp(uX_{\alpha})=\sum_{n=0}^{\infty}u^nX_{\alpha}^n\left/n!\right.\,
\end{equation}

We recall also that by~\cite[\S 10, Cor. (b) after
Theorem $29$]{St67}, every automorphism $\sigma$ of the Dynkin diagram
${\rm Dyn}(\dG)$
of $\dG$ can be extended to an automorphism of $\dG$ (and hence of
$\overline{\dG}$) and
$\dg$, still denoted by $\sigma$, which takes
$$
x_{\alpha}(u)\longrightarrow x_{\sigma(\alpha)}(\varepsilon_{\alpha}u)
\mbox{\ \
and\ \ }
X_{\alpha}\longrightarrow \varepsilon_{\alpha}X_{\sigma(\alpha)}.
$$
Here $\varepsilon_{\alpha}=\pm 1$ and if $\alpha\in \Pi$ then
$\varepsilon_{\alpha}=1$.
Thus we have a natural embedding
$$
{\rm Aut}({\rm Dyn}(\dG))
\hookrightarrow {\rm Aut}_{S-Lie}(L(\dg)_S).
$$

The group $\overline{\dG}(S)$ acts by $S$-automorphisms on $L(\dg)_S$ through
the adjoint representations $ad: \overline{\dG}\to
{\mathbf{GL}}(L(\dg)_S)$ and hence we also have a canonical embedding
$$
\overline{\dG}(S)\hookrightarrow {\rm Aut}_{S-Lie}(L(\dg)_S).
$$
As we said before, it is well-known (see \cite{P3} for example) that
$$
{\rm Aut}_{S-Lie}\,(L(\dg)_S)=\overline{\dG}\,(S)\rtimes {\rm Aut}({\rm Dyn}
(\dG)).
$$
For later use we need one more fact.

\begin{proposition} Let $f:\dG\to \overline{\dG}$ be the canonical
morphism.
The group $\overline{\dG}(S)$ is generated by the root subgroups
$f(U_{\alpha}), \ \alpha\in\Sigma$, and $\overline{{\mathbf T}}(S)$.
\end{proposition}
\begin{proof} Let ${\mathbf Z}\subset \dG$ be the center of $\dG$. The exact sequence
$$
1\longrightarrow \mathbf{Z} \longrightarrow \dG \longrightarrow
\overline{\dG}\longrightarrow 1
$$
gives rise to an exact sequence in Galois cohomology
$$
f(\dG(S))\hookrightarrow \overline{\dG}(S)
\longrightarrow {\rm Ker}\,[H^1(S,{\mathbf Z})\to H^1(S,\dG)] \longrightarrow 1.
$$
Since $H^1(S,{\mathbf Z})\to H^1(S,\dG)$ factors through
$$
H^1(S,{\mathbf Z})\longrightarrow H^1(S,{\mathbf T})\longrightarrow H^1(S,\dG)
$$
and since $H^1(S,{\mathbf T})=1$ (because ${\rm Pic}\,S=1$) we obtain
\begin{equation}\label{sequence1}
f(\dG(S))\hookrightarrow \overline{\dG}(S)
\longrightarrow H^1(S,{\mathbf Z})\longrightarrow 1.
\end{equation}

Similar considerations applied to
$$
1\longrightarrow {\mathbf Z} \longrightarrow {\mathbf T} \longrightarrow
\overline{{\mathbf T}}\longrightarrow 1
$$
show that
\begin{equation}\label{sequence2}
f({\mathbf T}(S))\hookrightarrow \overline{{\mathbf T}}(S) \longrightarrow
H^1(S,{\mathbf Z})\longrightarrow
1.
\end{equation}
The result now follows from (\ref{sequence1}) and (\ref{sequence2}).
\end{proof}

\begin{corollary}\label{generators1} One has
$$
{\rm Aut}_{S-Lie}\,(L(\dg)_S)=\langle\, {\rm Aut}({\rm Dyn}(\dG)),
\,U_{\alpha},\,
\alpha\in \Sigma,\, \overline{{\mathbf T}}(S)\,\rangle.
$$
\end{corollary}

\section{$k$-automorphisms of $L(\dg)_S$}

We keep the above notation.  Recall that for any algebra $\mathfrak{A}$
over a field $k$ the centroid
of $\mathfrak{A}$ is
$$
{\rm Ctrd}\,(\mathfrak{A}) =\{ \chi \in {\rm End}_{k}(\mathfrak{A})\ |
\chi (a\cdot b)
=a\cdot \chi(b)=\chi(a)\cdot b \mbox{\ \ for all \ } a,b\in\mathfrak{A}\,\}.
$$
It is easy to check that if $\chi_1,\chi_2\in {\rm Ctrd}(\mathfrak{A})$
then both linear operators
$\chi_1\circ \chi_2$ and $\chi_1+\chi_2$
are contained in
${\rm Ctrd}\,(\mathfrak{A})$ as well.
Thus, ${\rm Ctrd}\,(\mathfrak{A})$ is a unital associative
subalgebra of ${\rm End}_k(\mathfrak{A}).$ It is also well-known
that the centroid is commutative whenever $\mathfrak{A}$ is perfect.

\smallskip

\noindent
{\bf Example.} Consider the $k$-Lie algebra $\mathfrak{A}=L(\dg)_S$.
For any $s\in S$ the
linear $k$-operator $\chi_s: L(\dg)_S\to L(\dg)_S$ given by $x\to xs$ satisfy
$$\chi_s([x,y])=[x,\chi_s(y)]=[\chi_s(x),y],$$ hence $\chi_s\in {\rm Ctrd}\,
(L(\dg)_S)$.
Conversely, it is known (see \cite[Lemma $4.2$]{ABP}) that every element in
${\rm Ctrd}\,(L(\dg)_S)$
is of the form $\chi_s$. Thus,
$$
{\rm Ctrd}\,(L(\dg)_S)=\{\,\chi_s\ | \ s\in S\,\}\simeq S.
$$

\begin{proposition}{\rm (\cite[Proposition $1$]{P3})}\label{generatorsk-Lie}
One has
$$
\begin{array}{lll}
{\rm Aut}_{k-Lie}(L(\dg)_S) & \simeq  & {\rm Aut}_{S-Lie}(L(\dg)_S) \rtimes
{\rm Aut}_k
({\rm Ctrd}\,(L(\dg)_S))\\
& \simeq & {\rm Aut}_{S-Lie}(L(\dg)_S) \rtimes {\rm Aut}_k(S).
\end{array}
$$
\end{proposition}
\begin{corollary}\label{generators2}
One has
$$
{\rm Aut}_{k-Lie}(L(\dg)_S)=\langle\, {\rm Aut}_k\,(S),
 {\rm Aut}({\rm Dyn}(\dG)), \,U_{\alpha},\,
\alpha\in \Sigma,\, \overline{{\mathbf T}}(S)\,\rangle.
$$
\end{corollary}
\begin{proof} This follows from Corollary~\ref{generators1}
and Proposition~\ref{generatorsk-Lie}.
\end{proof}

\section{Automorphisms of $\widetilde{L}(\dg)_S$}


We remind the reader that the centre of $\widetilde{L}(\dg)_S$ is the
$k$-span of $c$
and that $\widetilde{L}(\dg)_S=L(\dg)_S\oplus kc$.
Since any automorphism $\phi$ of $\widetilde{L}(\dg)_S$ takes the centre into
itself we have a natural (projection)
mapping $$\mu:\widetilde{L}(\dg)_S \to \widetilde{L}(\dg)_S/kc\simeq
L(\dg)_S$$ which induces
the mapping $$\lambda: {\rm Aut}_{k-Lie}(\widetilde{L}(\dg)_S)\to
{\rm Aut}_{k-Lie}(L(\dg)_S)$$
given by $\phi\to \phi'$ where $\phi'(x)= \mu(\phi(x))$ for all $x\in L(\dg)_S$.
In the last formula we view $x$ as an element of $\widetilde{L}(\dg)_S$
through the embedding $L(\dg)_S\hookrightarrow \widetilde{L}(\dg)_S$.
\begin{remark}
It is straightforward to check
that $\phi'$ is indeed an automorphism of $L(\dg)_S$.
\end{remark}
\begin{proposition}\label{lambda}
The mapping $\lambda$
is an isomorphism.
\end{proposition}
\begin{proof} See \cite[Proposition $4$]{P3}.
\end{proof}

In what follows if
$\phi \in {\rm Aut}_{k-Lie}(L(\dg)_S)$ we  denote its (unique) lifting to
${\rm Aut}_{k-Lie}(\widetilde{L}(\dg)_S)$ by $\tilde{\phi}$.

\begin{remark}\label{lift}
For later use we need an explicit formula for lifts of
automorphisms of $L(\dg)_S$ induced by some ``special'' points
in $\overline{{\mathbf T}}(S)$ (those
which are
not in the image of ${\mathbf T}(S)\to \overline{{\mathbf T}}(S)$).
More precisely,
choose the decomposition $\overline{{\mathbf T}}\simeq \dG_{m,S}\times \cdots
\times \dG_{m,S}$ such that the canonical embedding $\dG_{m,S}\to
\overline{{\mathbf T}}$
into the $i$-th factor is the cocharacter
of $\overline{{\mathbf T}}$ dual to $\alpha_i$.
As usual, we have the decomposition
$\overline{{\mathbf T}}(S)\simeq \overline{{\mathbf T}}(k)\times
{\rm Hom}\,(\dG_m,\overline{{\mathbf T}})$.
The second factor in the last decomposition is the cocharacter lattice
of $\overline{{\mathbf T}}$ and its elements correspond (under the adjoint action)
to the subgroup in
$ {\rm Aut}_{S-Lie}(L(\dg)_S)$ isomorphic to ${\rm Hom}(Q,\mathbb{Z})$
where $Q$ is the corresponding root lattice:
if $\phi\in {\rm Hom}(Q,\mathbb{Z})$
it induces an $S$-automorphism of $L(\dg)_S$ (still denoted by $\phi$)
given by
$$X_{\alpha}\to X_{\alpha}\otimes s^{\phi(\alpha)},\ \
H_{\alpha_i}\to H_{\alpha_i}.$$
It is straightforward to check the mapping $\tilde{\phi}:\widetilde{L}(\dg)_S
\to \widetilde{L}(\dg)_S$ given by
$$
H_{\alpha}\to H_{\alpha}+\phi(\alpha)\langle X_{\alpha},X_{-\alpha}\rangle
\cdot c,\ \
H_{\alpha}\otimes s^p\to H_{\alpha}\otimes s^{p}
$$
if $p\not=0$ and
$$
X_{\alpha}\otimes s^p \to X_{\alpha}\otimes
s^{p+\phi(\alpha)}
$$
is an automorphism of $\widetilde{L}(\dg)_S$, hence it is the (unique)
lift of $\phi$.
\end{remark}

\section{Automorphisms of split affine Kac-Moody Lie algebras}\label{splitcase}

Since $\widetilde{L}(\dg)_S=[\widehat{L}(\dg)_S,\widehat{L}(\dg)_S]$
we have a natural (restriction) mapping
$$
\tau: {\rm Aut}_{k-Lie}\,(\widehat{L}(\dg)_S)\to
{\rm Aut}_{k-Lie}\,(\widetilde{L}(\dg)_S).$$
\begin{proposition}\label{surjectivity}
The mapping $\tau$ is surjective.
\end{proposition}
\begin{proof}
By Proposition~\ref{lambda} and Corollary~\ref{generators2}
the group ${\rm Aut}_{k-Lie}(\widetilde{L}(\dg)_S)$
has the distinguished system of generators $\{\,\tilde{\phi}\,\}$ where
$$\phi\in {\rm Aut}({\rm Dyn}(\dG)),\,
\overline{{\mathbf T}}(S),\,{\rm Aut}_k(S),\,
U_{\alpha}.$$ We want
to construct a mapping $\hat{\phi}:\widehat{L}(\dg)_S\to \widehat{L}(\dg)_S$
which preserves the identity
$$
[d,x\otimes t^{\frac{p}{m}}]_{\widehat{L}}=p\,x\otimes t^{\frac{p}{m}}
$$
for all $x\in \dg$ and whose restriction to $\widetilde{L}(\dg)_S$
coincides with $\tilde{\phi}$. These two properties would imply that
$\hat{\phi}$ is an automorphism of $\widehat{L}(\dg)_S$
lifting $\tilde{\phi}$.

If $\phi\in U_{\alpha}$ is unipotent we
define $\hat{\phi}$, as usual, through the
exponential map. If $\phi\in {\rm Aut}({\rm Dyn}(G))$ we put
$\hat{\phi}(d)=d$. If $\phi$ is as in Remark~\ref{lift}
we extend it by $d\to d-X$ where $X\in \dh$ is the unique element such that
$[X,X_{\alpha}]=\phi(\alpha)X_{\alpha}$ for all roots $\alpha\in \Sigma$.
Note that automorphisms of $L(\dg)_S$ given by points in
$\overline{{\mathbf T}}(k)$
are in the image of ${\mathbf T}(k)\to \overline{{\mathbf T}}(k)$ and
hence they are generated
by unipotent elements.
Lastly, if $\phi\in {\rm Aut}_k\,(S)$  is of the form $s\to
as^{-1}$ where $a\in k^{\times}$ (resp. $s\to as$) we extend
$\tilde{\phi}$ by
$\hat{\phi}(d)=-d$ (resp. $\hat{\phi}(d)=d$).
We leave it to the reader
to verify that in all cases $\hat{\phi}$ preserves the above identity and
hence $\hat{\phi}$
is an automorphism of $\widehat{L}(\dg)_S$.
\end{proof}

\begin{proposition}\label{subgroupV} One has ${\rm Ker}\,\tau\simeq V$ where
$V={\rm Hom}_k(kd,kc)$.
\end{proposition}
\begin{proof} We first embed $V\hookrightarrow
{\rm Aut}_{k-Lie}(\widehat{L}(\dg)_S)$.
Let $v\in V$. Recall that any element $x\in\widehat{L}(\dg)_S$ can be written
uniquely
in the form $x=x'+ad$ where $x'\in\widetilde{L}(\dg)_S$ and $a\in k$.
We define  $\hat{v}:\widehat{L}(\dg)_S
\to \widehat{L}(\dg)_S$ by $x\to x+v(ad)$. One checks that $\hat{v}$ is an
automorphism
of $\widehat{L}(\dg)_S$ and thus the required embedding
is given by $v\to \hat{v}$.

Since $\hat{v}(x')=x'$ for all
$x'\in \widetilde{L}$ we have
$\hat{v}\in {\rm Ker}\,\tau$.
Conversely, let $\psi\in {\rm Ker}\,\tau$. Then $\psi(x)=x$ for all $x\in
\widetilde{L}(\dg)_S$.
We need to show that  $\psi(d)=ac+d$ where $a\in k$.
Let $\psi(d)=x'+ac+bd$ where $a,b\in k$ and $x'\in L(\dg)_S$.
Since $[d,X_{\alpha}]_{\widehat{L}(\dg)_S}=0$ we get
$$[\psi(d),\psi(X_{\alpha})]_{\widehat{L}(\dg)_S}=0.$$ Substituting
$\psi(d)=x'+ac+bd$ we obtain $$[x'+ac+bd,X_{\alpha}]_{\widehat{L}(\dg)_S}=0$$ or
$[x',X_{\alpha}]_{\widetilde{L}(\dg)_S}=0$.
Since this is true for all roots $\alpha\in \Sigma$, the element $x'$ commutes
with $\dg$ and this can happen if and only if $x'=0$.

It remains to show that $b=1$. To see this we can argue similarly by considering
the equality
$$[d,X_{\alpha}\otimes t^{\frac{1}{m}}]_{\widehat{L}(\dg)_S}=X_{\alpha}\otimes
t^{\frac{1}{m}}$$
and applying $\psi$.
\end{proof}
\begin{corollary}\label{mainsequence}
The  sequence of groups
\begin{equation}\label{mainsequence1}
1\longrightarrow V \longrightarrow {\rm Aut}_{k-Lie}\,(\widehat{L}(\dg)_S)
\stackrel{\lambda \circ \tau}{\longrightarrow} {\rm Aut}_{k-Lie}\,(L(\dg)_S)
\longrightarrow 1
\end{equation}
is exact.
\end{corollary}

\section{Automorphism group of twisted affine Kac-Moody Lie algebras}


We keep the notation introduced in \S\,\ref{affinerealization}. In particular,
we fix an integer $m$ and a primitive root of unity $\zeta=\zeta_m\in k$ of
degree $m$.
Consider the $k$-automorphism $\zeta^{\times}: S \to S$ such that
$s\to \zeta s$ which we view as a $k$-automorphism of $L(\dg)_S$
through the embedding
$$
{\rm Aut}_k\,(S)\hookrightarrow
{\rm Aut}_{k-Lie}\,(L(\dg)_S)\simeq
{\rm Aut}_{S-Lie}\,(L(\dg)_S)\rtimes {\rm Aut}_k\,(S)
$$
(see Proposition~\ref{generatorsk-Lie}). As it is explained
in \S\,\ref{splitcase}
we then get the automorphism $\widehat{\zeta}^{\times}$ (resp.
$\widetilde{\zeta}^{\times}$)
of $\widehat{L}(\dg)_S$ (resp. $\widetilde{L}(\dg)_S$)
given by
$$
x\otimes s^i +ac+bd \longrightarrow x\otimes \zeta^i s^i+ac+bd
$$
where $a,b\in k$ and $x\in \dg$.

Consider now the abstract group $\Gamma=\mathbb{Z}/m\mathbb{Z}$ (which
can be identified with ${\rm Gal}\,(S/R)$ as already explained)
and define its action
on $\widehat{L}(\dg)_S$ (resp. $\widetilde{L}(\dg)_S,\, L(\dg)_S$)
with the use of
$\widehat{\zeta}^{\times}$ (resp. $\widetilde{\zeta}^{\times},\,
\zeta^{\times}$). More precisely,
for every $l\in\widehat{L}(\dg)_S$ we let
$\overline{i}(l):=(\widehat{\zeta}^{\times})^i(l)$. Similarly, we define
the action of $\Gamma$ on ${\rm Aut}_{k-Lie}\,(\widehat{L}(\dg)_S)$ by
$$
\overline{i}: {\rm Aut}_{k-Lie}\,(\widehat{L}(\dg)_S) \longrightarrow
{\rm Aut}_{k-Lie}\,(\widehat{L}(\dg)_S), \ \ x \to
(\widehat{\zeta}^{\times})^ix
(\widehat{\zeta}^{\times})^{-i}.
$$
Therefore, ${\rm Aut}_{k-Lie}\,(\widehat{L}(\dg)_S)$ can be viewed as
a $\Gamma$-set. Along the same lines one defines the action of $\Gamma$
on ${\rm Aut}_{k-Lie}\,(L(\dg)_S)$ and ${\rm Aut}_{S-Lie}(L(\dg)_S)$
with the use of
$\zeta^{\times}$.
It is easy to see that
$\Gamma$ acts trivially on the subgroup
$V\subset {\rm Aut}_{k-Lie}\,(\widehat{L}(\dg)_S)$
introduced in Proposition~\ref{subgroupV}.
Thus, $(\ref{mainsequence1})$ can be viewed as an exact
sequence of $\Gamma$-groups.

We next choose an element $\pi \in {\rm Aut}({\rm Dyn}(\dG))\subset
{\rm Aut}_k(\dg)$ of order $m$ (clearly, $m$ can take value $1,2$ or $3$
only).
Like before, we have the corresponding automorphism $\hat{\pi}$
of $\widehat{L}(\dg)_S$ given by
$$
x\otimes s^i +ac+bd \longrightarrow \pi(x)\otimes s^i+ac+bd
$$
where $a,b\in k$ and $x\in \dg$.

Note that $\widehat{\zeta}^{\times} \hat{\pi}=\hat{\pi}
\widehat{\zeta}^{\times}$.
It then easily follows that the assignment
$$
\overline{1} \to z_{\overline{1}}=\hat{\pi}^{-1}\in {\rm Aut}_{k-Lie}\,
(\widehat{L}(\dg)_S)
$$
gives rise to a cocycle  $z=(z_{\overline{i}})\in
Z^1(\Gamma,{\rm Aut}_{k-Lie}\,(\widehat{L}(\dg)_S))$.

This cocycle, in turn,
gives rise to a (new) twisted action of $\Gamma$ on $\widehat{L}(\dg)_S$ and
${\rm Aut}_{k-Lie}\,(\widehat{L}(\dg)_S)$.
Analogous considerations (with the use of $\pi$)
are applied to ${\rm Aut}_{k-Lie}\,(L(\dg)_S)$ and $L(\dg)_S$.
For  future reference note that $\hat{\pi}$ commutes with elements in $V$,
hence the twisted action of $\Gamma$ on $V$ is still trivial.
From now on we view
(\ref{mainsequence1}) as an exact sequence of $\Gamma$-groups, the action of $\Gamma$
being the twisted action.

\begin{remark}\label{twistedform}
As we noticed before the invariant subalgebra
$$
\mathcal{L}=L(\dg,\pi)=(L(\dg)_S)^{\Gamma}=
((\dg\otimes_kR)\otimes_R S)^{\Gamma}
$$
is a simple Lie algebra over $R$, a twisted form of a split Lie algebra
$\dg\otimes_k R$.
The same cohomological formalism also yields that
\begin{equation}\label{lifting}
{\rm Aut}_{R-Lie}\,(\mathcal{L})\simeq ({\rm Aut}_{S-Lie}\,(L(\dg)_S))^{\Gamma}.
\end{equation}
\end{remark}
\begin{remark}\label{worth} It is worth mentioning that 
the canonical embedding
$$
\begin{array}{lll}
\iota: ({\rm Aut}_{k-Lie}\,(L(\dg)_S))^{\Gamma} &\hookrightarrow &
{\rm Aut}_{k-Lie}\,((L(\dg)_S)^{\Gamma}) =
{\rm Aut}_{k-Lie}\,(\mathcal{L})
\simeq \\
& & {\rm Aut}_{R-Lie}\,(\mathcal{L})\rtimes {\rm Aut}_k\,(R),
\end{array}
$$
where the last isomorphism can be established
in the same way as in Proposition~\ref{generatorsk-Lie},
is not necessary surjective in general case. Indeed, one checks that if $m=3$
then the $k$-automorphism of $R$  given by 
$t\to t^{-1}$ and viewed as an element of 
${\rm Aut}_{k-Lie}\,(\mathcal{L})
\simeq
{\rm Aut}_{R-Lie}\,(\mathcal{L})\rtimes {\rm Aut}_k\,(R)$ is not in ${\rm Im}\,\iota$.
However (\ref{lifting}) implies that
the group ${\rm Aut}_{R-Lie}\,(\mathcal{L})$ is in the image of $\iota$.
\end{remark}
\begin{remark} The $k$-Lie algebra $\widehat{\mathcal{L}}=
(\widehat{L}(\dg)_S)^{\Gamma}$
is a twisted affine Kac--Moody Lie algebra. Conversely,
by the Realization Theorem 
every twisted affine Kac--Moody Lie algebra can be obtained in such a way.
\end{remark}
\begin{lemma}\label{cohomologyV} One has $H^1(\Gamma,V)=1$.
\end{lemma}
\begin{proof} Since $\Gamma$ is cyclic of order $m$ acting trivially
on $V\simeq k$ it follows that
$$
Z^1(\Gamma,V)=\{\,x\in k\ | \ mx=0\,\}=0
$$
as required.
\end{proof}

The long exact cohomological sequence associated to (\ref{mainsequence1})
together with Lemma~\ref{cohomologyV} imply the following.
\begin{theorem}\label{mainsequencetwist} The following sequence
$$
1 \longrightarrow V \longrightarrow ({\rm Aut}_{k-Lie}\,
(\widehat{L}(\dg)_S))^{\Gamma}
\stackrel{\nu}{\longrightarrow} ({\rm Aut}_{k-Lie}\,
(L(\dg)_S))^{\Gamma} \longrightarrow 1
$$
is exact. In particular, the group ${\rm Aut}_{R-Lie}\,(\mathcal{L})$
is in the image of the canonical mapping
$$
{\rm Aut}_{k-Lie}\,(\widehat{\mathcal{L}}) \longrightarrow {\rm Aut}_{k-Lie}
\,(\mathcal{L})
\simeq {\rm Aut}_{R-Lie}\,(\mathcal{L})\rtimes {\rm Aut}_k\,(R).$$
\end{theorem}
\begin{proof} The first assertion is clear. As for the second one, note
that as in Remark~\ref{worth} we have the canonical embedding
 $$
({\rm Aut}_{k-Lie}\,(\widehat{L}(\dg)_S))^{\Gamma} \hookrightarrow
{\rm Aut}_{k-Lie}\,((\widehat{L}(\dg)_S)^{\Gamma})=
{\rm Aut}_{k-Lie}\,(
\widehat{\mathcal{L}})
$$
and the commutative diagram
\[
\begin{CD}
({\rm Aut}_{k-Lie}\,(\widehat{L}(\dg)_S))^{\Gamma} @>\nu>>
({\rm Aut}_{k-Lie}\,(L(\dg)_S))^{\Gamma}\\
@VVV @VVV\\
{\rm Aut}_{k-Lie}\,(\widehat{\mathcal{L}}) @>>>
{\rm Aut}_{k-Lie}\,({\mathcal{L}})\\
\end{CD}
\]
Then surjectivity of $\nu$ and Remark~\ref{worth} yield the result.

\end{proof}

\section{Some properties of affine Kac-Moody
Lie algebras}\label{Some properties of affine Kac-Moody Lie algebras}

%
%

Henceforth we fix a simple finite dimensional Lie algebra $\dg$ and a
(diagram) automorphism $\sigma$ of finite order $m$. For brevity, we will
write $\widehat{\mathcal{L}}$ and $(\widetilde{\mathcal{L}},\mathcal{L})$ for
$\widehat{L}(\dg,\sigma)$ and $
(\widetilde{L}(\dg,\sigma),L(\dg,\sigma))$ respectively.

For all $l_1,l_2\in \mathcal{L}$ one has
\begin{equation}\label{difference}
[l_1,l_2] -[l_1,l_2]_{\widehat{\mathcal{L}}}=ac
\end{equation}
for some scalar $a\in k$. Using (\ref{newbracket}) it is also easy
to see that for all $y\in \mathcal{L}$ one has
\begin{equation}\label{differential}
[d,yt^n]_{\widehat{\mathcal{L}}}= mnyt^n+[d,y]_{\widehat{\mathcal{L}}}\,t^n
\end{equation}
\begin{remark} Recall that $\mathcal{L}$ has a natural $R$-module structure:
If $y=x\otimes t^{\frac{p}{m}}\in \mathcal{L}$ then
$$yt:=x\otimes t^{\frac{p}{m}+1} = x\otimes t^{\frac{p +m}{m}} \in \mathcal{L}.$$ Therefore since
$[d,y]_{\widehat{\mathcal{L}}}$ is contained in $\mathcal{L}$ the expression
$[d,y]_{\widehat{\mathcal{L}}}\,t^n$ is meaningful.
\end{remark}

The infinite dimensional Lie algebra $\widehat {\mathcal{L}}$ admits
a unique (up to non-zero scalar) invariant nondegenerate bilinear form
$(\cdot,\cdot)$.
Its restriction to
$\mathcal{L}\subset \widehat{\mathcal{L}}$
is nondegenerate (see \cite[7.5.1 and 8.3.8]{Kac}) and we have
$$(c,c)=(d,d)=0,\ 0\neq(c,d)=\beta\in k^{\times}$$ and
$$(c,l)=(d,l)=0 \mbox{\ \ for all\ } l\in \mathcal{L}.$$

\begin{remark}\label{restrictioninvariantform} It is known that a
nondegenerate invariant bilinear
form on $\widehat{\mathcal{L}}$ is unique up to nonzero scalar.
We may view $\widehat{\mathcal{L}}$ as a subalgebra in the split
Kac-Moody Lie algebra
$\widehat{L}(\dg)_S$. The last one also admits a nondegenerate
invariant bilinear
form and it is known  that its restriction to $\widehat{\mathcal{L}}$ is nondegenerate.
Hence this restriction is proportional to the form $(\cdot,\cdot)$.
\end{remark}

Let $\dh_{\overline{0}}$ be a Cartan subalgebra
of the Lie algebra $\dg_{\overline{0}}.$

\begin{lemma} \label{centralizer is Cartan}
 The centralizer of $\dh_{\overline{0}}$ in $\dg$ is a Cartan
 subalgebra $\dh$ of $\dg$.
\end{lemma}
\begin{proof}
See \cite[Lemma 8.1]{Kac}.
\end{proof}

 The algebra $\mathcal{H} = \dh_{\overline{0}} \oplus kc \oplus kd$ plays the
role of  Cartan subalgebra for $\widehat{\mathcal{L}}.$
 With respect to $\mathcal H$ our algebra $\widehat{\mathcal{L}}$ admits a
root space decomposition. The roots are of two types: anisotropic
(real) or isotropic (imaginary). This terminology comes from
transferring the form to $\mathcal{H}^*$ and computing the
``length" of the roots.

The {\it core} $\widetilde{\mathcal{L}}$ of $\widehat{ \mathcal{L}}$ is the subalgebra
generated by all the anisotropic roots. In our case we have
$\widetilde{\mathcal{L}}=\mathcal{L}\oplus kc$. The correct way to recover $\mathcal{L}$
inside $\widehat{\mathcal{L}}$ is as its core modulo its centre.\footnote{In
nullity one the core coincides with the derived algebra, but this
is not necessarilty true in higher nullities.}

If $\mm\subset \widehat{\mathcal{L}}$ is an abelian subalgebra and $\alpha\in
\mm^*={\rm Hom}(\mm,k)$ we denote the corresponding eigenspace in
$\widehat{\mathcal{L}}$ (with respect to the adjoint representation of
$\widehat{\mathcal{L}}$) by $\widehat{\mathcal{L}}_{\alpha}$. Thus,
$$
\widehat{\mathcal{L}}_{\alpha}=\{\,l\in \widehat{\mathcal{L}}\ | \
[x,l]_{\widehat{\mathcal{L}}}= \alpha(x)l \mbox{\ for all\ } x\in \mm\,\}.$$
The subalgebra $\mm$ is called {\it diagonalizable} in $\widehat{\mathcal{L}}$ if $$
\widehat{\mathcal{L}}=\bigoplus_{\alpha\in \mm^*} \widehat{\mathcal{L}}_{\alpha}.
$$

Every diagonalizable subalgebra of $\mm \subset \widehat{\mathcal{L}}$ is necessarily abelian.
We say that $\mm$ is a maximal (abelian) diagonalizable subalgebra (MAD) if
it is not properly contained in a larger diagonalizable
subalgebra of $\widehat{\mathcal{L}}$.
\begin{remark}
Every MAD of $\widehat{\mathcal{L}}$ contains the center $kc$ of $\widehat{\mathcal{L}}$.
\end{remark}


\begin{example}\label{example}
 The subalgebra $\mathcal H$ is a MAD in
$\widehat{\mathcal{L}}$ (see \cite[Theorem 8.5]{Kac}).
\end{example}


Our aim is to show that an arbitrary maximal diagonalizable
subalgebra $\mm\subset \widehat{\mathcal{L}}$ is conjugate to $\mathcal{H}$
under an element of ${\rm Aut}_k(\widehat{\mathcal{L}})$. For  future
reference we record the following facts:

\begin{theorem}\label{facts}

\noindent {\rm (a)} {\it Every diagonalizable subalgebra in $\mathcal{L}$ is
contained in a MAD of $\mathcal{L}$ and all MADs of $\mathcal{L}$
are conjugate. More
precisely, let $\dG$ be the simple simply connected
group scheme over $R$
corresponding to $\mathcal{L}$. Then for any MAD $\mm$ of $\mathcal{L}$
there exists
$g\in \dG(R)$ such that $Ad(g)(\mm)=\dh_{\overline{0}}$.}

\smallskip

\noindent {\rm (b)} {\it There exists a natural bijection between MADs of
$\widetilde{\mathcal{L}}$ and MADs of $\mathcal{L}$. Every diagonalizable
subalgebra
in $\widetilde{\mathcal{L}}$ is contained in a MAD of
$\widetilde{\mathcal{L}}$. All
MADs of $\widetilde{\mathcal{L}}$ are conjugate by elements in
$Ad(\dG(R))\subset {\rm Aut}_k(\mathcal{L})\simeq
{\rm Aut}_k(\widetilde{\mathcal{L}})$. }

\smallskip

\noindent {\rm (c)} {\it The image of the canonical map
${\rm Aut}_k(\widehat{\mathcal{L}}) \to {\rm Aut}_k(\widetilde{\mathcal{L}})
\simeq  {\rm Aut}_k({\mathcal{L}})$}
obtained by restriction to the derived subalgebra
$\widetilde{\mathcal{L}}$ contains ${\rm Aut}_{R-Lie}(\mathcal{L})$.

\end{theorem}
\begin{proof} (a) From the explicit realization of $\mathcal{L}$ one knows
that $\dh_{\overline{0}}$ is a MAD of $\mathcal{L}.$ Now (a) follows from \cite{CGP}.

(b) The correspondence follows from the fact that every MAD of
$\widetilde{\mathcal{L}}$ contains $kc.$ A  MAD $\widetilde{\mm}$ of
$\widetilde{\mathcal{L}}$ is necessarily of the form $\mm \oplus kc$ for some MAD $\mm$
of $\mathcal{L}$ and conversely. The canonical map
${\rm Aut}_k(\widetilde{\mathcal{L}}) \to {\rm Aut}_k(\mathcal{L})$ is an isomorphism
by Proposition~\ref{lambda}.


(c) This was established in Theorem \ref{mainsequencetwist}.

\end{proof}



\begin{lemma}\label{notproper}
If $\mm\subset \widehat{\mathcal{L}}$ is a MAD of $\widehat{\mathcal{L}}$ then
$\mm\not\subset \widetilde{\mathcal{L}}$.
\end{lemma}
\begin{proof} Assume that $\mm\subset
\widetilde{\mathcal{L}}$. By Theorem~\ref{facts} (b), there exists a MAD
$\mm'$ of
$\widetilde{\mathcal{L}}$
containing $\mm$. Applying again
Theorem~\ref{facts} we may assume that up to
conjugation by an element of ${\rm Aut}_k(\widehat{\mathcal{L}}),$ in fact of
$\widehat{\dG}(R),$ we have
$\mm\subset \mm'=\dh_{\overline{0}}\oplus kc$. Then $\mm$ is a proper subalgebra
of the MAD $\mathcal{H}$ of $\widehat{\mathcal{L}}$ and this contradicts the
maximality of $\mm$.
\end{proof}

In the next three sections we are going to prove some preliminary results
related to a subalgebra $\widehat{A}$ of the twisted affine Kac-Moody Lie algebra
$\widehat{\mathcal{L}}$ which
satisfies the following two conditions:
\medskip

\noindent a) {\it $\widehat{A}$ is of the form $\widehat{A}=A\oplus kc\oplus kd$,
where $A$ is an $R$-subalgebra of $\mathcal{L}$ such that $A\otimes_R K$ is a
semisimple Lie algebra over $K$ where $K=k(t)$ is
the fraction field of $R$.}
\smallskip

\noindent b) {\it The restriction to $\widehat{A}$ of the  non-degenerate invariant bilinear form
${\rm (-,-)}$ of  $\widehat{\mathcal{L}}$ is non-degenerate.}

\medskip

In particular, all these results will be valid for  $\widehat{A}=\widehat{\mathcal{L}}$.

\section{Weights of semisimple operators and their properties}

Let $x=x'+d\in\widehat{A}$ where
$x'\in A$. It induces a $k$-linear operator $$ad(x): \widehat{A}
\to \widehat{A},\ \ y \to ad(x)(y)=[x,y]_{\widehat{A}}.$$
We say that $x$
is a {$k$-diagonalizable element of}
$\widehat{A}$ if $\widehat{A}$ has a $k$-basis consisting
of eigenvectors of $ad(x)$. Throughout we assume that $x'\not=0$
and that $x$ is $k$-diagonalizable.

For any scalar $w\in k$ we let
$$
\widehat{A}_w=\{\,y\in \widehat{A}\ | \ [x,y]_{\widehat{A}}=wy\,\}.
$$
We say that $w$ is a {\it  weight} ($=$ eigenvalue) of $ad(x)$
if $\widehat{A}_w\not=0$.
More generally, if $O$ is a diagonalizable linear operator of a vector space $V$ over $k$ (of main interest to us are the vector
spaces $\widehat{A},\,\widetilde{A}=A\oplus kc,\,A$) and
if $w$ is its eigenvalue following standard practice we will denote
by $V_w\subset V$ the corresponding eigenspace of $O$.

\begin{lemma}\label{insider} (a) If $w$ is a nonzero weight of $ad(x)$
then $\widehat{A}_w
\subset \widetilde{A}$.

\smallskip

\noindent
(b) $\widehat{A}_0=\widetilde{A}_0
\oplus \langle\,x\,\rangle$.

\end{lemma}
\begin{proof} Clearly we have  $[\widehat{A},\widehat{A}]\subset\widetilde{A}$
and this implies
$ad(x)(\widetilde{A})\subset \widetilde{A}$.  It then follows that the linear operator
$ad(x)|_{\widetilde{A}}$ is
$k$-diagonalizable. Let
$\widetilde{A}=\oplus\, \widetilde{A}_{w'}$ where the sum is taken
over all weights of $ad(x)|_{\widetilde{A}}$. Since $x\in
 \widehat{A}_0$ and since
$\widehat{A}=\langle\,x\,\rangle \oplus \widetilde{A}$ we conclude
that
$$
\widehat{A}=\langle\,x, \widetilde{A}_0\,\rangle \oplus (\oplus_{w'\not=0}
\,\widetilde{A}_{w'}),
$$
so that the result follows.
\end{proof}

The operator $ad(x)|_{\widetilde{A}}$ maps the center $\langle\,c\,\rangle = kc$ of
$\widetilde{A}$ into itself, hence
it induces a linear operator $O_x$ of $A\simeq \widetilde{A}/kc$ which is also
$k$-diagonalizable. The last isomorphism is induced by a natural (projection)
mapping $\lambda:\widetilde{A} \to A$. If $w\not= 0$ the restriction of $\lambda$
to $\widetilde{A}_w$
is injective (because $\widetilde{A}_w$ does not contain $kc$).
Since $\widetilde{A}=\oplus_w \widetilde{A}_w$ it then follows that
$$
\lambda |_{\widetilde{A}_w}: \widetilde{A}_w\longrightarrow A_w$$
is an isomorphism for $w\not= 0$.
Thus the three linear operators $ad(x)$,
$ad(x)|_{\widetilde{A}}$ and $O_x$ have the same nonzero weights.

\begin{lemma}\label{shift}
Let $w\not=0$ be a weight of $O_x$ and let $n\in\mathbb{Z}$. Then
$w+mn$ is also a weight of $O_x$ and $A_{w+mn}=t^{n}A_{w}.$
\end{lemma}
\begin{proof}
Assume $y\in A_w\subset A$, hence $O_x(y)=wy.$ Let us show that $yt^n\in
A_{w+mn}.$ We have
\begin{equation}\label{weight}
O_x(yt^n)=
\lambda(ad(x)(yt^n))=\lambda([x,yt^n]_{\widehat{A}}).
\end{equation}

Substituting $x=x'+d$ we get
$$
[x,yt^n]_{\widehat{A}}=[x',yt^n]_{\widehat{A}}+[d,yt^n]_{\widehat{A}}
$$
Applying (\ref{difference}) and (\ref{differential}) we get that
the right hand side is equal to
$$
[x',y]\,t^n+ac+[d,y]_{\widehat{A}}\,t^n+mnyt^n
$$
where $a\in k$ is some scalar.
Substituting this into (\ref{weight}) we get
$$\label{calculations}
\begin{array}{lll}
O_x(yt^n) & = & \lambda([x',y]\,t^n+ac+[d,y]_{\widehat{A}}\,t^n+mnyt^n)\\
& = &
[x',y]\,t^n+\lambda([d,y]_{\widehat{A}}\,t^n)+mnyt^n
\end{array}
$$
By (\ref{difference}) there exists $b\in k$ such that
$$
[x',y]\,t^n=([x',y]_{\widehat{A}}+bc)\,t^n.
$$
Here we view $[x',y]\,t^n$ as an element in $\widehat{A}$. Therefore
$$
\begin{array}{lll}
O_{x}(yt^n) & = & mnyt^n+\lambda(([x',y]_{\widehat{A}}+bc)\,t^n+[d,y]_{\widehat{A}}\,t^n)\\
& = &
mnyt^n+\lambda(([x,y]_{\widehat{A}}+bc)\,t^n).\\
\end{array}
$$

We now note that by construction $[x,y]_{\widehat{A}}+bc$ is contained in $A\subset \widetilde{A}$.
Hence
$$
\lambda(([x,y]_{\widehat{A}}+bc)t^n)=\lambda([x,y]_{\widehat{A}}+bc)\,t^n=
\lambda([x,y]_{\widehat{A}}))\,t^n.$$
Since
$\lambda([x,y]_{\widehat{A}})=O_x(y)=wy$ we finally get
$$
O_{x}(yt^n)=mnyt^n+wyt^n=(w+mn)yt^n.
$$
Thus we have showed that $A_wt^n\subset A_{w+nm}$. By symmetry
$A_{w+nm}t^{-n}\subset A_w$ and we are done.
\end{proof}

We now consider the case $w=0$.
\begin{lemma}\label{w=0} Assume that ${\rm dim}\,\widetilde{A}_0>1$
and $n\in\mathbb{Z}$. Then  $mn$ is a weight of $ad(x)$.
\end{lemma}
\begin{proof} Since ${\rm dim}\,\widetilde{A}_0>1$ there exists nonzero
$y\in A$ such that $[x,y]_{\widehat{A}}=0$.
Then the same computations as above show that $[x,yt^n]_{\widetilde{A}}=mnyt^n$.
\end{proof}

Our next aim is to show that if $w$ is a weight of $ad(x)$ so is
$-w$.
We remind  the reader that $\widehat{A}$ is equipped
with the nondegenerate invariant bilinear form $(-,-)$. Hence for all
$y,z\in\widehat{A}$ one has
\begin{equation}\label{invariantform}
([x,y]_{\widehat{A}},z)=-(y,[x,z]_{\widehat{A}}).
\end{equation}
\begin{lemma}\label{opposite weight}
 If $w$ is a weight of $ad(x)$ then so is $-w$.
\end{lemma}
\begin{proof} If $w=0$ there is nothing to prove. Assume $w\not=0$.
Consider the root space decomposition
$$
\widehat{A}=\bigoplus_{w'} \widehat{A}_{w'}.
$$
It suffices to show
that for any two weights $w_1,w_2$ of $ad(x)$ such that $w_1+w_2\not=0$
the subspaces $\widehat{A}_{w_1}$ and $\widehat{A}_{w_2}$ are orthogonal
to each other.
Indeed, the last
implies that if $-w$ were not a weight then every element in $\widehat{A}_w$
would be orthogonal to all elements in $\widehat{A},$ which is impossible.

Let $y\in \widehat{A}_{w_1}$ and $z\in \widehat{A}_{w_2}$.
Applying (\ref{invariantform})
we have
$$
w_1(y,z)=([x,y]_{\widehat{A}},z)=
-(y,[x,z]_{\widehat{A}})=-w_2(y,z).
$$
Since $w_1\not=-w_2$ we conclude
$(y,z)=0.$
\end{proof}

Now we switch our interest to the operator $O_x$ and its weight subspaces.
Since the nonzero weights of $ad(x)$, $ad(x)|_{\widetilde{A}}$ and $O_x$ are
the same we obtain,
by Lemmas~\ref{shift} and \ref{w=0},
that for every weight $w$ of $O_x$ all elements in the set
$$
\{\,w+mn\ | \ n\in \mathbb{Z}\,\}
$$
are also weights of $O_x$. We call this set of weights
by {\it $w$-series}. Recall that by Lemma~\ref{shift} we have
$$
A_{w+mn}=A_wt^n.
$$
\begin{lemma}\label{Rspan} Let $w$ be a weight of $O_x$ and let
$A_w R$ be the $R$-span of $A_w$ in $A$. Then the natural map
$\nu:A_w\otimes_k R\rightarrow A_wR$ given by $l\otimes t^n\mapsto
lt^n$ is an isomorphism of $k$-vector spaces.
\end{lemma}
\begin{proof}
Clearly, the sum $\sum_n A_{w+mn}$ of vector subspaces $A_{w+mn}$ in
$A$ is a direct sum. Hence
\begin{equation}\label{isomorphism2}
 A_wR=\sum_{n} A_wt^n=\sum_{n} A_{w+mn}=
\bigoplus_{n}A_{w+mn}
\end{equation}
Fix a $k$-basis $\{e_i\}$ of $A_w$. Then $\{e_i\otimes t^j\}$ is
a $k$-basis of $A_w\otimes_k R$. Since $$\nu(e_i\otimes t^n)=e_it^n
\in A_{w+mn}$$
the injectivity of $\nu$ easily follows from (\ref{isomorphism2}).
The surjectivity is also obvious.
\end{proof}

\noindent
{\bf Notation:} We will denote the $R$-span $A_w R$  by $A_{\{w\}}$.

\smallskip

By our construction $A_{\{w\}}$ is an $R$-submodule of
$A$ and
\begin{equation}\label{basis}
A=\bigoplus_{w} A_{\{w\}}
\end{equation}
where the sum is taken over fixed representatives of weight series.
\begin{corollary}\label{finiteness1}  ${\rm dim}_k\, A_w< \infty$.
\end{corollary}
\begin{proof} Indeed, by the above lemma we have
$$
{\rm dim}_k\, A_w={\rm rank}_R\,(A_w\otimes_kR)=
{\rm rank}_R\,A_wR= {\rm rank}_R\, A_{\{w\}} \leq {\rm rank}_R \,A < \infty,
$$
as required.
\end{proof}
\begin{corollary}\label{finiteness} There are finitely many weight series.
\end{corollary}
\begin{proof} This follows from the fact that $A$ is a free
$R$-module of finite rank.
\end{proof}
\begin{lemma}\label{sumtwoweights}
Let $w_1,w_2$ be
weights of $O_x$. Then $[A_{w_1},A_{w_2}]\subset
A_{w_1+w_2}.$
\end{lemma}
\begin{proof} This is straightforward to check.
\end{proof}

\section{Weight zero subspace}

\begin{theorem}\label{zeroweight}  $A_0\not=0$.
\end{theorem}
\begin{proof} Assume that $A_0=0$. Then, by Lemma~\ref{shift},
 $A_{mn}=0$ for all
$n\in \mathbb{Z}$. It follows that for any weight $w$, any integer $n$
and all $y\in A_w$, $z\in A_{-w+mn}$ we have $[y,z]=0.$
Indeed
\begin{equation}\label{commuting}
[A_w,A_{-w+mn}]\subset A_{w+(-w)+mn}=A_{mn}=0.
\end{equation}

For $y\in A$ the operator $ad(y):A\to A$ may be viewed
as a $k$-operator or as an $R$-operator. When we deal with
the Killing form $\langle-,-\rangle$ on the $R$-Lie algebra $A$
we will view $ad(y)$ as an $R$-operator
of $A$.

\begin{lemma}\label{orthogonality1} Let $w_1,\,w_2$ be weights of
$ad(x)$ such that
 $\{w_1\}\not=\{-w_2\}$.
Then for any integer $n$ and all $y\in A_{w_1}$ and $z\in A_{w_2+mn}$
we have $\langle\,y,z\,\rangle=0$.
\end{lemma}
\begin{proof}  Let $w$ be a weight of $ad(x)$. By our condition we have
$\{w\}\not=\{ w+w_1+w_2\}$. Since
$(ad(y)\circ ad(z))(A_{\{w\}})\subset A_{\{w+w_1+w_2\}}$, in any
$R$-basis of $A$ corresponding to the decomposition
(\ref{basis}) the operator $ad(y)\circ ad(z)$ has zeroes on the diagonal,
hence ${\rm Tr}\,(ad(y)\circ ad(z))=0$.
\end{proof}
\begin{lemma}\label{orthogonality2}
Let $w$ be a weight of $ad(x)$, $n$ be an integer  and let
$y\in A_w$. Assume that $ad(y)$ viewed as an $R$-operator of $A$ is nilpotent.
Then for every $z\in A_{-w+mn}$ we have $\langle\,y,z\,\rangle=0$.
\end{lemma}
\begin{proof} Indeed, let $l$ be such that $(ad(y))^l=0$. Since by
(\ref{commuting}),
$ad(y)$ and $ad(z)$ are commuting operators we have
$$
(ad(y)\circ ad(z))^l=(ad(y))^l\circ (ad(z))^l=0.
$$
Therefore $ad(y)\circ ad(z)$ is  nilpotent and this implies
its trace is zero.
\end{proof}

Since the Killing form is nondegenerate, it follows immediately from the
above two lemmas
that for every nonzero element $y\in A_w$
the operator $ad(y)$ is not nilpotent. Recall that by Lemma~\ref{sumtwoweights}
we have $ad(y)(A_{w'})\subset A_{w+w'}$. Hence taking into consideration
Corollary~\ref{finiteness} we conclude that there exits a weight $w'$
and a positive integer $l$ such that
$$
ad(y)(A_{\{w'\}})\not=0,\ (ad(y)\circ ad(y))(A_{\{w'\}})\not=0,\ldots,
(ad(y))^l(A_{\{w'\}})\not=0
$$
and $(ad(y)^l(A_{\{w'\}})\subset A_{\{w'\}}$.
We may assume that $l$ is the smallest positive integer satisfying
these conditions.
Then all consecutive scalars
\begin{equation}\label{consecutive}
w',\,w'+w,\,w'+2w,\ldots,w'+lw
\end{equation}
are
weights of $ad(x)$, $\{w'+iw\}\not=\{w'+(i+1)w\}$ for $i<l$ and
$\{w'\}=\{w'+lw\}$.
In particular, we automatically get that
$lw$ is an integer (divisible by $m$) which in turn implies that $w$ is
a rational number.

Thus,
under our
assumption $A_0=0$ we have proved that all weights of $ad(x)$ are
rational numbers.
We now choose  (in a unique way) representatives
$w_1,\ldots,w_s$ of all weight series such that $0< w_i<m$ and up to renumbering
we may assume that  $$0<w_1< w_2< \cdots < w_s< m.$$

\begin{remark}
Recall that for any weight $w_i$, the scalar $-w_i$ is also a weight.
Since $0< -w_i+m < m$
the representative of the  weight series $\{-w_i\}$
is $m-w_i$. Then the inequality $m-w_i\geq w_1$ implies $m-w_1\geq w_i$.
Hence out of necessity  we have $w_s=m-w_1$.
\end{remark}

We now apply the observation (\ref{consecutive}) to the weight $w=w_1$.
Let $w'=w_i$ be  as in (\ref{consecutive}). Choose the integer $j\geq 0$
such that $w_i+jw_1$, $w_i+(j+1)w_1$  are weights and
$w_i+jw_1< m$, but $w_i+(j+1)w_1\geq m$. We note that
since $m$ is not a weight of $ad(x)$
we automatically obtain $w_i+(j+1)w_1> m$. Furthermore,
we have $w_i+jw_1\leq w_s=m-w_1$
(because $w_i+jw_1$
is a weight of $ad(x)$). This implies $$
m < w_i+(j+1)w_1\leq w_s+w_1=m-w_1+w_1=m
$$
-- a contradiction  that completes the proof of the theorem.
\end{proof}

\section{A lower bound of dimensions of MADs in $\widehat{\mathcal{L}}$}

\begin{theorem}\label{lower bound of dimension theorem} Let $\gm\subset
\widehat{\mathcal{L}}$
be a MAD. Then
${\rm dim}\,\gm\geq 3$.
\end{theorem}

By Lemma~\ref{notproper},
$\gm$ contains an element $x$
of the form $x=x'+d$ where $x'\in \mathcal{L}$ and it also contains $c$.
Since $x$ and $c$ generate
a subspace of $\gm$ of dimension $2$
the statement of the theorem is equivalent to
$\langle\,x,c\,\rangle\not=\gm$.

Assume the contrary: $\langle\,x,c\,\rangle=\gm$.
Since $\gm$ is $k$-diagonalizable we have the weight
 space decomposition
$$
\widehat{\mathcal{L}}=\bigoplus_{\alpha} \widehat{\mathcal{L}}_{\alpha}
$$
where the sum is taken over linear mappings
$\alpha\in \gm^*={\rm Hom}\,(\gm,k)$.
To find a contradiction
we first make some
simple observations about the structure of the corresponding eigenspace
$\widehat{\mathcal{L}}_{0}$.

If $\widehat{\mathcal{L}}_{\alpha}\not=0$, it easily follows
that $\alpha(c)=0$ (because $c$ is in the center of $\widehat{\mathcal{L}}$).
Then $\alpha$ is determined uniquely by the value $w=\alpha(x)$
and so instead of $\widehat{\mathcal{L}}_{\alpha}$ we will write
$\widehat{\mathcal{L}}_w$.

Recall that
by Theorem~\ref{zeroweight}, $\mathcal{L}_0\not=0$.
Our aim is first to show that $\mathcal{L}_0$ contains a nonzero
element $y$ such
that the adjoint operator $ad(y)$ of $\mathcal{L}$ is $k$-diagonalizable.
We will next see that $y$  necessarily
commutes with $x$ viewed as an element in $\widehat{\mathcal{L}}$ and that
it is $k$-diagonalizable in $\widehat{\mathcal{L}}$ as well. It then
follows that
the subspace in $\widehat{\mathcal{L}}$ spanned by $c,\,x$ and $y$
is a commutative
$k$-diagonalizable subalgebra and this contradicts the fact that
$\gm$ is a MAD.

\begin{lemma}~\label{x commutes with y in A^} Let $y\in\mathcal{L}$
be nonzero such that $O_x(y)=0$.
Then $[x,y]_{\widehat{\mathcal{L}}}=0$.
\end{lemma}
\begin{proof}
Assume that $[x,y]_{\widehat{\mathcal{L}}}=bc\not=0$. Then
$$
(x,[x,y]_{\widehat{\mathcal{L}}})=(x,bc)=(x'+d,bc)=(d,bc)=\beta b\not=0.
$$
On the other hand, since the form is invariant we get
$$
(x,[x,y]_{\widehat{\mathcal{L}}})=([x,x]_{\widehat{\mathcal{L}}},y)=(0,y)=0
$$
-- a contradiction which completes the proof.
\end{proof}
\begin{lemma}~\label{diagonalizable element in A^} Assume that $y\in
\mathcal{L}_0$ is nonzero
and that the adjoint
operator $ad(y)$ of $\mathcal{L}$ is $k$-diagonalizable. Then $ad(y)$
viewed as an operator of $\widehat{\mathcal{L}}$ is also $k$-diagonalizable.
\end{lemma}
\begin{proof} Choose a $k$-basis $\{\,e_i\,\}$ of $\mathcal{L}$
consisting of eigenvectors
of $ad(y)$. Thus we have $[y,e_i]=u_ie_i$ where $u_i\in k$ and hence
$$
[y,e_i]_{\widehat{\mathcal{L}}}=u_ie_i+b_ic
$$
where $b_i\in k$.

\smallskip

\noindent
{\it Case $1$}: Suppose first that $u_i\not=0$. Let $$
\tilde{e}_i=e_i+\frac{b_i}{u_i}\cdot c \in \widetilde{\mathcal{L}}.
$$
Then we have
$$
[y,\tilde{e}_i]_{\widehat{\mathcal{L}}}=[y,e_i]_{\widehat{\mathcal{L}}}=
u_ie_i+b_ic=u_i \tilde{e}_i
$$
and therefore  $\tilde{e}_i$ is an eigenvector of the operator
$ad(y):\widehat{\mathcal{L}}\to \widehat{\mathcal{L}}$.

\smallskip

\noindent
{\it Case $2$}:
Let now $u_i=0$. Then $[y,e_i]_{\widehat{\mathcal{L}}}=b_ic$ and we claim that
$b_i=0.$ 
Indeed, we have
$$
(x,[y,e_i]_{\widehat{\mathcal{L}}})=([x,y]_{\widehat{\mathcal{L}}},e_i)=(0,e_i)=0
$$
and on the other hand
$$
(x,[y,e_i]_{\widehat{A}})=(x,b_ic)=(x'+d,b_ic)=(d,b_ic)=\beta b_i.
$$
It follows that $b_i=0$ and thus $\tilde{e_i}=e_i$ is an eigenvector of $ad(y)$.

Summarizing, replacing $e_i$ by $\tilde{e}_i$ we see that
the set $\{\,\tilde{e}_i\,\}\cup \{\,c,x\,\}$ is a $k$-basis
of $\widehat{\mathcal{L}}$
consisting of eigenvectors of $ad(y)$.
\end{proof}
\begin{proposition}~\label{diagonalizable element in A} The subalgebra
$\mathcal{L}_0$
contains an element $y$ such that the operator $ad(y):\mathcal{L}\to
\mathcal{L}$ is $k$-diagonalizable.
\end{proposition}

\begin{proof} We split the proof in three steps.

\smallskip

\noindent
{\it Step $1$}:
Assume first that there exists $y\in \mathcal{L}_0$  which as an element in
$\mathcal{L}_K=\mathcal{L}\otimes_R K$ is semisimple. We claim that our
operator $ad(y)$ 
is $k$-diagonalizable.
Indeed, choose representatives $w_1=0,w_2,\ldots,w_l$ of the weight series
of $ad(x)$.
The sets $\mathcal{L}_{w_1},\ldots,\mathcal{L}_{w_l}$ are vector spaces
over $k$ of finite dimension, by Lemma~\ref{finiteness1},
and they are stable with respect
to $ad(y)$ (because $y\in \mathcal{L}_0$). In each $k$-vector
space $\mathcal{L}_{w_i}$ choose a Jordan basis $$\{e_{ij},\ j=1,\ldots,l_i\}$$
of the operator $ad(y)|_{\mathcal{L}_{w_i}}$. Then
the set
\begin{equation}\label{R-basis}
\{\,e_{ij}, \ i=1,\cdots,l,\ j=1,\ldots, l_i\,\}
\end{equation}
is an $R$-basis
of $\mathcal{L}$, by Lemma~\ref{Rspan} and the decomposition
given in (\ref{basis}).
It follows that the matrix of the operator $ad(y)$
viewed as a $K$-operator of $\mathcal{L}\otimes_R K$
is a block diagonal matrix whose blocks corresponds to
the matrices of $ad(y)|_{\mathcal{L}_{w_i}}$ in the basis $\{e_{ij}\}$.
Hence (\ref{R-basis}) is a Jordan basis for $ad(y)$ viewed as an operator on
$\mathcal{L}\otimes_R K$. Since $y$ is a semisimple element of
$\mathcal{L}\otimes_R K$ all  matrices of $ad(y)|_{\mathcal{L}_{w_i}}$
are diagonal and this in turn implies that $ad(y)$ is
$k$-diagonalizable operator of $\mathcal{L}$.

\smallskip

\noindent
{\it Step $2$}: We next consider the case when all elements in $\mathcal{L}_0$
viewed as elements of the $R$-algebra $\mathcal{L}$
are nilpotent. Then $\mathcal{L}_0$, being finite dimensional,
is a nilpotent Lie algebra over $k$.
In particular its center is nontrivial since $\mathcal{L}_0\not=0$.
Let $c\in \mathcal{L}_0$ be a nonzero central element of
$\mathcal{L}_0.$  For any $z\in \mathcal{L}_0$ the
operators $ad(c)$ and $ad(z)$ of
$\mathcal{L}$ commute. Then $ad(z)\circ ad(c)$ is nilpotent,
hence  $\langle c,z\rangle =0.$
Furthermore, by Lemma~\ref{orthogonality1} $\langle c,z\rangle =0$ for any
$z\in \mathcal{L}_{w_i}$, $w_i\neq 0$. Thus $c\neq 0$ is
in the radical of the Killing
form of $\mathcal{L}$ -- a contradiction.

\smallskip

\noindent
{\it Step $3$}: Assume now that $\mathcal{L}_0$ contains an element $y$ which as
an element of $\mathcal{L}_K$ has nontrivial semisimple part $y_s$.
Let us first show that $y_s\in \mathcal{L}_{\{0\}}\otimes_R K$ and then
that $y_s\in \mathcal{L}_0$. By Step $1$, the last would complete the proof of
the proposition.

By  decomposition (\ref{basis}) applied to $A=\mathcal{L}$ we may write $y_s$
as a sum
$$
y_s=y_1+y_2+\cdots +y_l
$$
where $y_i\in \mathcal{L}_{\{w_i\}}\otimes_R K$. In Step $1$ we showed
that in an appropriate $R$-basis (\ref{R-basis}) of $\mathcal{L}$ the matrix of
$ad(y)$ is block diagonal whose blocks correspond
to the Jordan matrices of $ad(y)|_{\mathcal{L}_{w_i}}: \mathcal{L}_{w_i}\to
\mathcal{L}_{w_i}$.
It follows
that the semisimple part of $ad(y)$ is also a block diagonal matrix whose
blocks are semisimple parts of $ad(y)|_{A_{w_i}}$.

Since $\mathcal{L}_K$  is a semisimple Lie algebra
over a perfect field we get that $ad(y_{s})=ad(y)_s$.
Hence for all weights $w_i$ we have
\begin{equation}\label{semisimplepart}
[y_s,\mathcal{L}_{w_i}]\subset
\mathcal{L}_{w_i}.
\end{equation}
On the other hand, for any $u\in
\mathcal{L}_{w_i}$ we have
$$
ad(y_s)(u)=[y_1,u]+ [y_2,u]+\cdots +[y_l,u].
$$
Since $[y_j,u]\in \mathcal{L}_{\{w_i+w_j\}}\otimes_R K$,
it follows that $ad(y_s)(u)\in \mathcal{L}_{\{w_i\}}$
if and only if $[y_2,u]=\cdots=[y_l,u]=0$. Since this is true
for all $i$ and all $u\in \mathcal{L}_{w_i}$ and since the kernel
of the adjoint representation  of $\mathcal{L}_K$ is trivial
we obtain $y_2=\cdots=y_l=0$.
Therefore $y_s\in \mathcal{L}_{\{0\}}\otimes_R K$.

It remains to show that $y_s\in \mathcal{L}_0$. We may write $y_s$
in the form
$$
y_s=\frac{1}{g(t)}(u_0\otimes 1+ u_1\otimes t+\cdots+u_m\otimes t^m)
$$
where $u_0,\cdots,u_l\in \mathcal{L}_0$ and $g(t)=g_0+g_1t+\cdots+g_nt^n$
is a polynomial with coefficients $g_0,\ldots,g_n$ in $k$ with $g_n\not=0$.
The above equality can be rewritten in the form
\begin{equation}\label{newform}
g_0y_s+g_1y_s\otimes t+\cdots+g_ny_s\otimes t^n=u_0\otimes 1+
\cdots+u_m\otimes t^m.
\end{equation}

Consider an arbitrary index $i$ and let $u\in \mathcal{L}_{w_i}$.
Recall that by (\ref{semisimplepart}) we have
$$ad(y_s)(\mathcal{L}_{w_i})\subset \mathcal{L}_{w_i}.$$ Applying
both sides of (\ref{newform}) to $u$ and comparing
$\mathcal{L}_{w_i+n}$-components we conclude that
$[g_ny_s,u]=[u_n,u]$. Since this is true for all $u$ and all $i$ and since
the adjoint representation of $\mathcal{L}_K$ has trivial kernel
we obtain $g_ny_s=u_n$. Since $g_n\not=0$ we get
$y_s=u_n/g_n\in \mathcal{L}_0$.
\end{proof}

Now we can easily finish the proof of
Theorem ~\ref{lower bound of dimension theorem}.
Suppose the contrary.
Then ${\rm dim}(\gm)<3$ and hence by Lemma
~\ref{notproper} we have $\gm=\langle c,x^{'}+d\rangle$ with
$x^{'}\in \mathcal{L}$. Consider the operator $O_x$ on $\mathcal{L}$.
By Theorem~\ref{zeroweight} we have $\mathcal{L}_0\neq 0$.
By Propositions~\ref{diagonalizable element in A}
and ~\ref{diagonalizable element in A^} there exists a nonzero
$k$-diagonalizable element $y\in \mathcal{L}_0$. Clearly, $y$ is not contained
in $\gm$.
Furthermore, by Lemma~\ref{x commutes with y in A^},
 $y$ viewed as an element of $\widehat{\mathcal{L}}$ commutes with $\gm$
and by Lemma~\ref{diagonalizable element in A^} it is $k$-diagonalizable
in $\widehat{\mathcal{L}}$.
It follows that the subspace $\gm_{1}=\gm\oplus \langle y\rangle$ is
an abelian $k$-diagonalizable subalgebra of $\widehat{\mathcal{L}}.$
But this contradicts  maximality of $\gm$.

\section{All MADs are conjugate}

\begin{theorem}\label{main} Let $\widehat{\dG}(R)$ be the preimage of
$\{Ad(g) : g \in \dG(R)\}$ under the canonical map
${\rm Aut}_k(\widehat{\mathcal{L}}) \to {\rm Aut}_k(\mathcal{L}).$
Then all MADs of $\widehat{\mathcal{L}}$ are conjugate 
under $\widehat{\dG}(R)$ to the subalgebra $\mathcal{H}$ 
in \ref{example}.
\end{theorem}
\begin{proof}
Let $\gm$ be a $MAD$ of $\widehat{\mathcal{L}}.$ By Lemma~\ref{notproper},
$\gm\not\subset \widetilde{\mathcal{L}}$. Fix a vector $x=x'+d\in \gm$ where
$x'\in \mathcal{L}$ and let $\gm'=\gm\cap \mathcal{L}$. Thus we have
$\gm=\langle\,x,c,\gm'\,\rangle$. Note that $\gm'\not=0$, by Theorem
\ref{lower bound of dimension theorem}. Furthermore, since $\gm'$ is
$k$-diagonalizable in $\mathcal{L},$ without loss of generality we may assume
that $\gm'\subset \dh_{\overline{0}}$ given that by Theorem \ref{facts}(b)
there exists $g\in
\dG(R)$ such that $Ad(g)(\mm')\subset \dh_{\overline{0}}$
and that by Theorem~\ref{mainsequencetwist} $g$ has lifting to
${\rm Aut}_{k-Lie}(\widehat{\mathcal{L}})$.

Consider the weight space decomposition
\begin{equation}\label{rootdecomposition}
\mathcal{L}=
\mathop{\oplus}_{i}\, L_{\alpha_{i}}
\end{equation}
with respect to the $k$-diagonalizable subalgebra $\gm'$ of $\mathcal{L}$
where $\alpha_i\in (\gm')^*$
and as usual
$$
L_{\alpha_i}=\{\, z\in \mathcal{L}\ |\ [t,z]=\alpha_i(t)z \mbox{\ for all\ }
t\in \gm'\}.$$

\begin{lemma}\label{g_1 tens R invar otnositelno O_x}
$L_{\alpha_i}$
is invariant with respect to the operator $O_x$.
\end{lemma}
\begin{proof}
The $k$-linear operator $O_x$ commutes with $ad(t)$ for all $t\in \gm'$
(because $x$ and $\gm'$ commute in $\widehat{\mathcal{L}}$), so
the result follows.
\end{proof}

\begin{lemma}\label{in centralizer} We have
$x^{'}\in L_0.$
\end{lemma}
\begin{proof}
By our construction $\gm'$ is contained in $ \dh_0$,
hence $d$ commutes with the elements of $\gm'$.
But $x$ also commutes with the elements of $\gm'$ and so does $x'=x-d$.
\end{proof}

$L_0=C_{\mathcal{L}}(\gm')$, being
the Lie algebra of the reductive group scheme $C_{\dG}(\gm')$
(see \cite{CGP}), is of the form
$L_0={\bf z}\oplus A$ where ${\bf z}$ and $A$ are the Lie algebras of
the central torus of $C_{\dG}(\gm')$ and its semisimple part respectively.
Our next goal is to show that $A=0$.

Suppose this is not true. To get a contradiction we will show that the subset
$\widehat{A}=A\oplus kc\oplus kd\subset \widehat{\mathcal{L}}$
is a subalgebra satisfying conditions
a) and b) stated at the end of
\S\,\ref{Some properties of affine Kac-Moody Lie algebras}
and that it is stable with respect to $ad(x)$. This, in turn,
will allow us to construct an
element
$y\in A$ which viewed as an element of ${\widehat{\mathcal{L}}}$
commutes with $x$ and $\gm'$ and is $k$-diagonalizable. The last, of course,
contradicts the maximality of
$\gm$.

Let $\dH$ denote
the simple simply connected Chevalley-Demazure  algebraic
$k$-group corresponding to  $\dg$.
Since $\dG$ is split over $S$ we have
$$\dH_S = \dH \times_k S \simeq\dG_S = \dG \times_R S.$$ Let
$C_{\dg}(\gm')
={\bf t} \oplus{\bf r}$ where
${\bf t}$
is the Lie algebra of the central torus of the reductive $k$-group
$C_{\dH}(\gm')$ and
${\bf r}$ is the Lie algebra of its semisimple part.
Since  centralizers commute with base change, we obtain that
$$
\dt_S=\dt\otimes_k S ={\bf z}\otimes_R S={\bf z}_S,\ \
{\bf r}_S={\bf r}\times_k S=A\otimes_R S=A_S.$$
\begin{lemma}\label{invariantunderadd}
We have $ad(d)({A})  \subset
{A}$ and in particular $\widehat{A}$ is a subalgebra of $\widehat{\mathcal{L}}$.
\end{lemma}
\begin{proof}
Since ${\bf r}$ consists of ``constant'' elements we have
$[d,{\bf r}]_{\widehat{L}(\dg)_S}=0$, and this implies that
$[d,{\bf r}_S]_{\widehat{L}(\dg)_S}\subset {\bf r}_S$.
Also, viewing ${\mathcal{L}}$
as a subalgebra of $\widehat{L}(\dg)_S$ we have
$[d,\mathcal{L}]_{\widehat{\mathcal{L}}}\subset \mathcal{L}$.
Furthermore, $S/R$ is faithfully flat, hence
$A=A_S\cap \mathcal{L}={\bf r}_S\cap\mathcal{L}$.
Since  both subalgebras ${\bf r}_S$ and $\mathcal{L}$
are stable with respect to $ad(d)$, so is their intersection.
\end{proof}

\begin{lemma}\label{nondegencentralizer}
The restriction of the nondegenerate invariant bilinear form $(\cdot,\cdot)$
on $\widehat{\mathcal{L}}$ to $L_0$ is nondegenerate.
\end{lemma}
\begin{proof}
We mentioned before that the restriction of
$(\cdot,\cdot)$ to $\mathcal{L}$ is nondegenerate.
Hence in view of decomposition (\ref{rootdecomposition}) it
suffices to show that for all $a\in L_0$ and $b\in L_{\alpha_{i}}$
with $\alpha_i\not=0$
we have $(a,b)=0$.

Let
$l\in \gm'$ be such that $\alpha_i(l)\neq 0.$ Using the
invariance of $(\cdot,\cdot)$ we get $$
\alpha_i(l)(a,b)=(a,\alpha_i(l)b)=(a,[l,b])=([a,l],b)=0.
$$
Hence $(a,b)=0$ as required.
\end{proof}

\begin{lemma}\label{nondegeneracyofinvariantform}
The restriction of  $(\cdot,\cdot)$
to $A$ is  nondegenerate.
\end{lemma}
\begin{proof}
By lemma(\ref{nondegencentralizer}) it is enough to show that
${\bf z}$ and $A$ are orthogonal in $\widehat{\mathcal{L}}$.
Moreover, viewing ${\bf z}$ and $A$ as subalgebras of the split
affine Kac-Moody Lie
algebra $\widehat{L}(\dg)_S$ and using Remark~\ref{restrictioninvariantform}
we conclude that it suffices to verify that ${\bf z}_S={\bf t}_S$ and
$A_S={\bf r}_S$ are orthogonal in $\widehat{L}(\dg)_S$.

Let $a\in {\bf t}$ and $b\in {\bf r}$. We know that
$$(at^{\frac{i}{m}},bt^{\frac{j}{m}})=\langle a,b\rangle \delta_{i+j,0}$$
where $\langle \cdot,\cdot \rangle$ is a Killing form
of $\dg$. Since ${\bf r}$ is a semisimple algebra we have
${\bf r}=[{\bf r},{\bf r}]$. It follows that we can write $b$ in the form
$b=\sum[a_i,b_i]$ for some $a_i,b_i\in {\bf r}$. Using the facts  that
${\bf t}$ and ${\bf r}$ commute and that the Killing form is invariant
we have
$$
\langle a,b\rangle =\langle a,\sum[a_i,b_i]\rangle=\sum\langle
[a,a_i],b_i\rangle=
\sum\langle 0,b_i\rangle =0.$$
Thus $(at^{\frac{i}{m}},bt^{\frac{j}{m}})=0$.
\end{proof}

\begin{lemma}\label{stable}
The $k$-subspace  $A\subset \mathcal{L}$ is invariant with respect to $O_x$.
\end{lemma}
\begin{proof}
Let $a\in A.$ We need to verify that
$$[x,a]_{\widehat{\mathcal{L}}}\in
A\oplus kc\subset \widehat{\mathcal{L}}.
$$
But $[d,A]_{\widehat{\mathcal{L}}}\subset A+kc$ by
Lemma~\ref{invariantunderadd}.
We also have
$$[x',A]_{\widehat{\mathcal{L}}}\subset
A\oplus kc$$ (because $x'\in L_0$, by Lemma~\ref{in centralizer},
and $A$ viewed as a
subalgebra in $L_0$ is an ideal). Since $x=x'+d$ the result follows.
\end{proof}

According to Lemma~\ref{in centralizer} we can write
$x'=x'_0+x'_1$ where $x'_0\in {\bf z}$ and $x'_1\in
A$.

\begin{lemma}\label{ox=ox'}
We have $O_x|_{A}=O_{x^{'}_{1}+d}|_A$. In particular, the operator
$O_{x^{'}_1+d}|_A$ of $A$ is $k$-diagonalizable.
\end{lemma}
\begin{proof} By Lemma~\ref{stable}, we have
$O_x(A)\subset A$. Since $O_x$ is $k$-diagonalizable (as
an operator of $\mathcal{L}$), so is the operator $O_x|_{A}$ of
$A$.
Therefore the last assertion of the lemma follows from the first
one.

Let now $a\in A$. Using the fact that $x_0'$ and $a$ commute in $\mathcal{L}$
we have
$$
[x',a]_{\widehat{\mathcal{L}}}=[x'_0,a]_{\widehat{\mathcal{L}}}+
[x'_1,a]_{\widehat{\mathcal{L}}}=[x'_1,a]_{\widehat{\mathcal{L}}}+bc
$$
for some $b\in k$. Thus $O_x(a)=O_{x_1'+d}(a)$.
\end{proof}

\begin{lemma}\label{diagonalizable element in s^}
 The operator $ad(x'_1+d):\widehat{A}\to \widehat{A}$ is $k$-diagonalizable.
\end{lemma}
\begin{proof}
Since by Lemma~\ref{ox=ox'} $O_{x'_1+d}|_A:A\to A$ is
$k$-diagonalizable we can apply the same arguments as in
Lemma~\ref{diagonalizable element in A^}.
\end{proof}

Now we can produce the required element $y$.
It follows from Lemma~\ref{nondegeneracyofinvariantform} that the Lie algebra
$\widehat{A}$ satisfies all the conditions stated at the end of
Section~\ref{Some properties of affine Kac-Moody Lie algebras}.
By Lemma~\ref{diagonalizable element in s^}, $ad(x'_1+d)$ is
$k$-diagonalizable operator of $\widehat{A}$. Hence arguing as in
Theorem~\ref{lower bound of dimension theorem} we see that there exists a
nonzero
$y\in A$ such that $[y,x'_1+d]_{\widehat{\mathcal{L}}}=0$ and $ad(y)$ is a
$k$-diagonalizable operator on $\widehat{A}$. Then by
Lemma~\ref{ox=ox'} we have
$O_x(y)=O_{x'_1+d}(y)=0$ and hence, by
Lemma~\ref{x commutes with y in A^}, $x$ and $y$ commute in
${\widehat{\mathcal{L}}}.$

According to our plan it remains
to show that $y$ is $k$-diagonalizable in $\widehat{\mathcal{L}}$.
To see this we need
\begin{lemma} Let $z\in \gm'$. Then $[z,y]_{\widehat{\mathcal{L}}}=0$.
\end{lemma}
\begin{proof}
Since
$
y\in A\subset C_{\mathcal{L}}(\gm')
$
we have $[z,y]_{\mathcal{L}}=0$. Then $[z,y]_{\widehat{\mathcal{L}}}=bc$
for some $b\in k.$ It follows
$$
0=(0,y)=([x,z]_{\widehat{\mathcal{L}}},y)=(x,[z,y]_{\widehat{\mathcal{L}}})=(x^{'}+d,bc)
=(d,bc)=\beta b.
$$
This yields $b=0$ as desired.
\end{proof}

\begin{proposition}\label{niceproperty}
The operator $ad(y):\widehat{\mathcal{L}}\to \widehat{\mathcal{L}}$ is
k-diagonalizable.
\end{proposition}
\begin{proof}
According to Lemma~\ref{diagonalizable element in A^}, it suffices
to prove that $ad(y):\mathcal{L}\to \mathcal{L}$ is $k$-diagonalizable.
Since $y$
viewed as an element of $A$ is semisimple it is still semisimple viewed as an
element of $\mathcal{L}$. In particular, the $R$-operator
$ad(y):\mathcal{L}\to\mathcal{L}$ is also
semisimple.

Recall that we have the decomposition of $\mathcal{L}$ into the direct sum of
the weight spaces with respect to $O_x:$
$$
\mathcal{L}=\bigoplus_w
\mathcal{L}_w=\bigoplus_i \bigoplus_n \mathcal{L}_{w_i+mn}=\bigoplus_i
\mathcal{L}_{\{w_i\}}.
$$
Since $y$ and $x$ commute in $\widehat{\mathcal{L}}$,  for all
weights $w$ we have $ad(y)(\mathcal{L}_w)\subset \mathcal{L}_w$. If we choose
any $k$-basis of $\mathcal{L}_w$ it is still an $R$-basis of
$\mathcal{L}_{\{w\}}=\mathcal{L}_w\otimes_k R$ and
in this basis the $R$-operator $ad(y)|_{\mathcal{L}_{\{w\}}}$
and the $k$-operator $ad(y)|_{\mathcal{L}_w}$ have the same matrices.
Since the $R$-operator $ad(y)|_{\mathcal{L}_{\{w\}}}$ is semisimple, so is
$ad(y)|_{\mathcal{L}_w}$, i.e. $ad(y)|_{\mathcal{L}_w}$  is a
$k$-diagonalizable operator.
Thus $ad(y):\mathcal{L}\to\mathcal{L}$ is $k$-diagonalizable.
\end{proof}

Summarizing, assuming $A\not=0$ we have constructed the
$k$-diagonalizable element
$$
y\not\in \gm=\langle\,\gm',x,c\,\rangle
$$
in $\widehat{\mathcal{L}}$ which commutes with $\gm'$ and $x$ in
$\widehat{\mathcal{L}}$. Then the subalgebra $\langle\,\gm,y\,\rangle$ in
$\widehat{\mathcal{L}}$ is commutative and $k$-diagonalizable which is
impossible since $\gm$ is a MAD. Thus
$A$ is necessarily trivial and this implies
$C_{\mathcal{L}}(\gm')$ is the Lie algebra of the
$R$-torus $C_{\dG}(\gm')$, in particular  $C_{\mathcal{L}}(\gm')$ is abelian.

Note that $x'\in C_{\mathcal{L}}(\gm')$, by Lemma~\ref{in centralizer},
and that $\dh_{\overline{0}}\subset C_{\mathcal{L}}(\gm')$
(because $\gm'\subset \dh_{\overline{0}}$, by construction). Since
$C_{\mathcal{L}}(\gm')$ is abelian and since $x=x'+d$ it follows that
$ad(x)(\dh_{\overline{0}})=0$. Hence $\langle\,\dh_{\overline{0}},x,c\,\rangle$
is a commutative $k$-diagonalizable subalgebra in $\widehat{\mathcal{L}}$.
But it contains our MAD $\gm$. Therefore
$\gm=\langle\,\dh_{\overline{0}},x,c\,\rangle$.
To finish the proof of
Theorem~\ref{main} it now suffices to show that $x'\in \dh_{\overline{0}}$.
For that, in turn, we may view $x'$ as an element of  $L(\dg)_S$
and it suffices to show that $x'\in \dh$ because $\dh\cap \mathcal{L}=
\dh_{\overline{0}}$.
\begin{lemma}
 $x'\in \dh$.
\end{lemma}
\begin{proof}
Consider the root space decomposition of $\dg$ with respect
to the Cartan subalgebra $\dh$:
$$
\dg=\dh \oplus(\mathop{\oplus}_{\alpha\not=0}\, \dg_{\alpha}).
$$
Every $k$-subspace $\dg_{\alpha}$ has dimension $1$.
Choose a nonzero elements $X_{\alpha}\in \dg_{\alpha}$.
It follows from $\gm'=\dh_{\overline{0}}$ that $C_{L(\dg)_S}(\gm')=\dh_S$.
Thus  $x'\in \dh_S$. 
Then $\dg_{\alpha}\otimes_k S$ is stable with respect to $ad(x')$ and
clearly it is stable with respect to $ad(d)$. Hence it is also
stable with respect
to $O_x$.

Arguing as in Lemma~\ref{shift} one can easily see that the operator
$O_x$, viewed as an operator of $L(\dg)_S$, is $k$-diagonalizable.
 Since $\dg_{\alpha}\otimes_k S$ is stable with respect to $O_x$,
 it is the direct sum of its weight subspaces. Hence
 $$\dg_{\alpha}\otimes_k S=\mathop{\oplus}_w (L(\dg)_S)_{\{w\}}$$
 where $\{w\}=\{w+j/m \ |\ j\in\mathbb{Z}\}$ is the
weight series corresponding to
 $w$.
 But  $\dg_{\alpha}\otimes_k S$ has rank $1$ as an $S$-module.
 This implies that in the above decomposition we have only one weight series
 $\{w\}$
 for some weight $w$ of $O_x$.

 We next note that automatically we have ${\rm dim}_k(L(\dg)_S)_{w}=1$.
 Any its nonzero vector which is a generator
 of the $S$-module $\dg_{\alpha}\otimes_k S$ is
of the form $X_{\alpha}t^{\frac{j}{m}}$.
 It follows from Lemma~\ref{shift} that $\dg_{\alpha}=\langle X_{\alpha}\rangle$
 is also a weight subspace of $O_x$. Thus for every root $\alpha$ we have
 $$
 [x,X_{\alpha}]_{\widehat{L}(\dg)_S}=
 [x'+d,X_{\alpha}]_{\widehat{L}(\dg)_S}=[x',X_{\alpha}]=b_{\alpha}X_{\alpha}
 $$
 for some scalar $b_{\alpha}\in k$. Since $x'\in \dh_S$ this can happen
 if and only if $x'\in \dh$.
\end{proof}

By the previous lemma we have $x'\in \dh_{\overline{0}}$, hence
$$\gm=\langle\,\dh_{\overline{0}},c,d\,\rangle=
\mathcal{H}.$$ The proof of
Theorem~\ref{main} is complete.
\end{proof}

\end{document}